\documentclass[12pt]{article}

\usepackage{amsmath}
\usepackage{amsthm}
\usepackage{amsfonts}
\usepackage{upgreek, amssymb}
\usepackage{graphicx}
\usepackage{enumitem}
\usepackage{lmodern}
\usepackage{subcaption}

\usepackage{diagbox}
\usepackage{tikz}
\usepackage[ruled, lined, linesnumbered]{algorithm2e}

\usepackage{url}
\usepackage{hyperref}\hypersetup{colorlinks=true}

\newtheorem{thm}{Theorem}[section]
\newtheorem{lem}[thm]{Lemma}

\newtheorem{cor}[thm]{Corollary}
\newtheorem{fact}[thm]{Fact}
\newtheorem{conjecture}[thm]{Conjecture}
\theoremstyle{definition}
\newtheorem{definition}[thm]{Definition}
\theoremstyle{remark}
\newtheorem{remark}[thm]{Remark}

\numberwithin{equation}{section}

\newcommand{\be}{\begin{equation}}
\newcommand{\ee}{\end{equation}}
\newcommand{\bea}{\begin{eqnarray}}
\newcommand{\eea}{\end{eqnarray}}
\newcommand{\nn}{\nonumber}

\newcommand\R{\mathbb{R}}

\newcommand\N{\mathbb{N}}

\newcommand\ext{\text{ext}}

\newcommand{\bp}{\text{bp}}
\newcommand{\cp}{\text{cp}}
\newcommand{\cc}{\text{cc}}

\title{Regarding two conjectures on clique and biclique partitions}
\author{Dhruv Rohatgi \hspace{10pt} John C. Urschel \hspace{10pt} Jake Wellens}
\begin{document} 
\maketitle

\begin{abstract}
For a graph $G$, let $\cp(G)$ denote the minimum number of cliques of $G$ needed to cover the edges of $G$ exactly once. Similarly, let $\bp_k(G)$ denote the minimum number of bicliques (i.e. complete bipartite subgraphs of $G$) needed to cover each edge of $G$ exactly $k$ times. We consider two conjectures -- one regarding the maximum possible value of $\cp(G) + \cp(\overline{G})$ (due to de Caen, Erd\H{o}s, Pullman and Wormald) and the other regarding $\bp_k(K_n)$ (due to de Caen, Gregory and Pritikin). We disprove the first, obtaining improved lower and upper bounds on $\max_G \cp(G) + \cp(\overline{G})$, and we prove an asymptotic version of the second, showing that $\bp_k(K_n) = (1+o(1))n$. 

\end{abstract}

\section{Introduction}

For a fixed family of graphs $\mathcal{F}$, an $\mathcal{F}$-partition of a graph $G$ is a collection $\mathcal{C} = \{H_1, \dots, H_\ell\}$ of subgraphs $H_i \subset G$ such that each edge of $G$ belongs to exactly one $H_i \in \mathcal{C}$, and each $H_i$ is isomorphic to some graph in $\mathcal{F}$. When $\mathcal{F} = \{K_r\}_{r \geq 2}$, we refer to $\mathcal{F}$-partitions as \emph{clique partitions}, and when $\mathcal{F} = \{K_{s, t}\}_{s, t \geq 1}$, the corresponding partitions are called \emph{biclique partitions}. The size $|\mathcal{C}|$ of the smallest clique partition of $G$ is called the \emph{clique partition number} of $G$, denoted $\cp(G)$. The \emph{biclique partition number} $\bp(G)$ is defined analogously. Both $\cp(G)$ and $\bp(G)$ (and their many variants) are NP-hard to compute in general graphs, but have been studied extensively from a combinatorial perspective, in part because of their connections to various areas of computer science (see, e.g. \cite{Jukna2006}). In this paper, we consider two longstanding combinatorial questions related to these quantities.

\subsection[]{Biclique partitions of $K_n$}

In 1971, Graham and Pollak \cite{Graham1971} showed that, for every $n \geq 2$,
\be
\bp(K_n) = n - 1.
\ee
In particular, the edges of $K_n$ can be partitioned into $n-1$ stars $$K_{1, n-1}, K_{1, n-2}, \dots, K_{1,1}$$ centered at different vertices, while the corresponding lower bound holds by an elegant linear algebraic argument. The lower bound argument easily generalizes to give 
\be\label{k_lowerbound}
\bp_k(K_n) \geq n -1
\ee
for any $k$, where $\bp_k(G)$ is the size of the smallest collection $\{H_1, \dots, H_\ell\}$ of bicliques $H_i \subset G$ such that each edge of $G$ belongs to $H_i$ for exactly $k$ different values of $i \in [\ell]$. As a matter of notation, such a collection is called a \emph{$k$-biclique cover} of $G$. More generally, a \emph{$\{k_1,\dots,k_t\}$-biclique cover} of $G$ is a collection $\{H_1,\dots,H_\ell\}$ of bicliques $H_i \subset G$ such that for each edge of $G$ there is some $k \in \{k_1,\dots,k_t\}$ such that the edge belongs to exactly $k$ of the bicliques.

In 1993, de Caen, Gregory and Pritikin conjectured that (\ref{k_lowerbound}) is tight for sufficiently large $n$:

\begin{conjecture}[de Caen et al.~\cite{deCaen1993}]\label{deCaen conj}
For every positive integer $k$, $$\emph{\bp}_{k}(K_n) = n-1$$ for all sufficiently large $n$.
\end{conjecture}
The same authors prove their conjecture for each $k \leq 18$, using special constructions from design theory \cite{deCaen1993}. However, the best-known upper bound for general $k$ is $\bp_{k}(K_n) = O(kn)$, obtained by simply compounding a small-$k$ construction.

In Section \ref{section:bp}, we show that, to leading order, Conjecture \ref{deCaen conj} is true.

\begin{thm}
For every positive integer $k$, $$\emph{\bp}_{k}(K_n) = \left(1 + o(1)\right) n.$$
\end{thm}
More precisely, we construct a family of designs (inspired by classical ideas of Nisan and Wigderson \cite{Nisan1989}), that yields a $k$-covering of $K_n$ by at most $n + 2kn^{3/4} + k\sqrt{n} $ complete bipartite subgraphs.



\subsection[]{Clique partitions of $G$ and $\overline{G}$}

In 1986, de Caen, Erd\H{o}s, Pullman and Wormald \cite{DEPW} investigated the maximum value of $\cp(G) + \cp(\overline{G})$ over the set $\mathcal{G}_n$ of all graphs $G$ on $n$ vertices, and proved that
\be\label{upper and lower}
\frac{7n^2}{25} + O(n) \leq \max_{G \in \mathcal{G}_n}\, \cp(G) + \cp(\overline{G}) \leq \frac{13n^2}{30} + O(n).\ee
They conjectured that the lower bound $\frac{7n^2}{25}$ is tight up to $o(n^2)$ terms, and left closing the gap in (\ref{upper and lower}) as an open problem.\footnote{In the same paper \cite{DEPW}, the authors solve the corresponding problem for $\cc(G) + \cc(\overline{G})$, where $\cc(G)$ is the minimal number of cliques in $G$ needed to cover every edge \textit{at least} once, showing that $\max_{G \in \mathcal{G}_n} \cc(G) + \cc(\overline{G}) =  \frac{n^2}{4}(1+o(1))$. This is tight up to the $o(1)$ error by $K_{n/2, n/2}$, and the error term was later removed by Pyber \cite{Pyber} for $n > 2^{1500}$.} 
\begin{conjecture}[de Caen et al.~\cite{DEPW}]\label{DEPW_conj}
$$\max_{G \in \mathcal{G}_n} \, \emph{\cp}(G) + \emph{\cp}(\overline{G}) \sim \frac{7}{25}n^2.$$
\end{conjecture}
In Section \ref{sec:lower_bd}, we show that the family of graphs constructed in \cite{DEPW} can actually be modified to improve the lower bound in (\ref{upper and lower}), thereby disproving Conjecture \ref{DEPW_conj}.
\begin{thm}\label{thm_lower}
For infinitely many $n$, there exists a self-complementary graph $G \in \mathcal{G}_n$ with $\emph{\cp}(G) \geq \frac{23}{164}n^2 + o(n^2)$.
\end{thm}
The upper bound in (\ref{upper and lower}) essentially comes from greedily selecting edge-disjoint triangles from $G$ and $\overline{G}$, forming clique partitions into $K_3$'s and $K_2$'s. Subsequent work on complementary triangle packings, first by Erd\H{o}s et al. \cite{EFGJL} and later by Keevash and Sudakov \cite{KS}, improved significantly upon the greedy packing, with the latter authors showing the existence of a packing with $\frac{n^2}{12.89}$ edge-disjoint triangles. The resulting clique partitions (as observed by Bujtas et al. \cite{BDGT}) contain a total of $0.34481n^2 + o(n^2)$ cliques, improving the $0.4\overline{3}n^2$ upper bound in (\ref{upper and lower}). However, partitions into triangles and edges can never push this bound below $0.3\overline{3}n^2$, as illustrated by $G = K_{n/2, n/2}$. 
In Section \ref{sec:upper_bd}, we extend the ideas of Keevash and Sudakov to the complementary clique partition problem, improving (\ref{upper and lower}) beyond the limits of triangle packings:
\begin{thm}
For all $G \in \mathcal{G}_n$, $\emph{\cp}(G) + \emph{\cp}(\overline{G}) \leq 0.3186n^2 + o(n^2)$.
\end{thm}

\section[]{A $k$-biclique covering of $K_n$}\label{section:bp}

Our goal in this section is to construct a collection of $(1+o(1))n$ bicliques on a set of $n$ vertices such that all $\binom{n}{2}$ edges belong to exactly $k$ bicliques in the collection. We recall the definition of a \emph{combinatorial design} in the sense of Nisan and Wigderson \cite{Nisan1989} from their classical paper on pseudorandom generators.

\begin{definition}
A family of sets $\{S_1,\dots,S_n\}$ with $S_1,\dots,S_n \subseteq [d]$ is a $(n,d,t,m)$-design if:
\begin{enumerate}
\item $|S_i| = m$ for all $i \in [n]$;
\item $|S_i \cap S_j| \leq t$ for all $i,j \in [n]$ with $i \neq j$.
\end{enumerate}
\end{definition}

We construct our designs in a way that differs from \cite{Nisan1989} and better suits our particular choice of parameters:

\begin{lem}\label{lemma:design}
For any positive integers $m$ and $t$, there exists some $N$ such that an $(n,d,t,m)$-design with $d \leq 2mn^{1/(t+1)}$ exists for all $n \geq N$.
\end{lem}

\begin{proof}
Let $N$ be large enough that there are at least $m$ prime numbers in the interval $[n^{1/(t+1)}, 2n^{1/(t+1)}]$ for every $n \geq N$; this is possible by the Prime Number Theorem. Fix some $n \geq N$, and choose $m$ distinct primes $p_1,\dots,p_m$ in this interval. We will pick sets $S_1,\dots,S_n$ from the disjoint union $$U = \bigsqcup_{k=1}^m \mathbb{Z}/p_k \mathbb{Z}.$$ For $i \in [n]$, let $S_i$ consist of $m$ elements from $U$, one from each group. Specifically, for $k \in [m]$, pick element $i \pmod{p_k}$ from group $\mathbb{Z}/p_k\mathbb{Z}$.

It is clear that $|S_i| = m$ for all $i$, and that $d := |U| = \sum_{k=1}^m p_k \leq 2mn^{1/(t+1)}$. We claim that $|S_i \cap S_j| \leq t$ for all distinct $i, j \in [n]$. Indeed, suppose to the contrary that $|S_i \cap S_j| > t$ for some distinct $i,j \in [n]$. Then among the chosen primes, there are $t+1$ primes $p_{l_1},\dots,p_{l_{t+1}}$ with $i \equiv j \pmod{p_{l_k}}$ for each $k \in [t+1]$. But then $$\prod_{k=1}^{t+1} p_{l_k}  \bigm | (i-j).$$ Since $i \neq j$, it follows that $$|i-j| \geq \prod_{k=1}^{t+1} p_{l_k} \geq n,$$ a contradiction.
\end{proof}

\begin{remark}
The above design is in fact optimal up to constant factors. Consider any $(n,d,t,m)$-design, where the sets are contained in a universe $U$ of size $d$. For every $(t+1)$-element subset of $U$, there is at most one set among $S_1,\dots,S_n$ that contains the subset. Since each $S_i$ contains $\binom{m}{t+1}$ subsets of size $t+1$, we must have $\binom{d}{t+1} \geq n \binom{m}{t+1},$ so $$d \geq \left (n \binom{m}{t+1} (t+1)!\right)^{1/(t+1)}  \geq \frac{1}{e} n^{1/(t+1)} m.$$
\end{remark}

We will only use the special case $(n, k\sqrt{n}, 1, \lfloor k/2 \rfloor)$ of Lemma~\ref{lemma:design}, which we state explicitly below as a corollary.

\begin{cor}\label{cor:design} For any positive integer $k$ and all $n$ sufficiently large, there is some $d \leq k\sqrt{n}$ and sets $S_1,\dots,S_n \subseteq [d]$ with $|S_i| = \lfloor k/2 \rfloor$ for all $i$, and $|S_i \cap S_j| \leq 1$ for all $i \neq j$. \\
\end{cor}

We also require a result of Alon \cite{Alon1997} on $\{1,2\}$-biclique coverings of $K_n$, which are collections of bicliques such that every edge of $K_n$ belongs to either 1 or 2 of the bicliques in the collection. The size of the smallest such collection is denoted $\bp_{\{1,2\}}(K_n)$.

\begin{fact}
[Alon, \cite{Alon1997}]\label{Alon} For all $n$, $\emph{\bp}_{\{1,2\}}(K_n) \leq 2\sqrt{n}$.
\end{fact}

Finally, we construct a $k$-biclique covering of $K_n$.

\begin{thm}\label{thm:bicliqueconstruction}
Let $k$ be a positive integer. Then for all sufficiently large $n$, $$\emph{\bp}_{k}(K_n) \leq n  + 2kn^{3/4} + k\sqrt{n}.$$
\end{thm}

\begin{proof}
Let $n$ be large enough to apply Corollary \ref{cor:design}. Let $S_1, \dots, S_n \subseteq [d]$ be the sets in the resulting design, and let $p_1, \dots, p_{\lfloor k/2\rfloor}$ be the corresponding primes used in the proof of Lemma \ref{lemma:design}. Define bicliques $B_1,\dots,B_d \subset K_n$ by letting $B_i$ be the biclique between $\{j \in [n] \mid i \in S_j\}$ and $\{j \in [n] \mid i \not \in S_j\}$. Then any edge $\{i,j\}$ is covered exactly $|S_i| + |S_j| - 2|S_i \cap S_j|$ times, and this number is equal to either $2\lfloor k/2 \rfloor$ or $2 \lfloor k/2 \rfloor - 2$ (depending on whether $|S_i \cap S_j| = 0$ or 1).

If $k$ is odd, every edge still needs to be covered either $1$ or $3$ more times. Let us define a \emph{triple-edge} to be an edge $\{i,j\}$ with $|S_i \cap S_j| = 1$. An edge $\{i,j\}$ is a triple-edge if and only if there exists some index $l$ and remainder $r$ such that $i \equiv j \equiv r \pmod{p_l}$. We can define a clique $C_{l,r}$ consisting of all vertices $i$ with $i \equiv r \pmod{p_l}$. Observe that every triple-edge is contained in exactly one such clique, and every such clique contains only triple-edges. To make progress, we will construct a $\{1,2\}$-biclique covering of each clique $C_{l ,r}$. The number of cliques $C_{l,r}$ is at most $k\sqrt{n}$, and each has size at most $\sqrt{n}$, so by Fact \ref{Alon}, at most $k\sqrt{n} \cdot 2n^{1/4} = 2kn^{3/4}$ bicliques are needed to $\{1,2\}$-cover every clique $C_{l,r}$. Now every edge needs to be covered $1$ or $2$ more times.

If $k$ is even, every edge needs to be covered only $0$ or $2$ more times, so we skip the above step. Finally, in either case, we'll ``pad'' the covering so that every edge is covered exactly $k$ times. To do this, define bicliques $D_1,\dots,D_n$ where $D_i$ is the star centered at vertex $i$ and containing edges to all vertices $j < i$ such that $\{i,j\}$ needs to be covered $1$ or $2$ more times, and to all vertices $j > i$ such that $\{i,j\}$ needs to be covered $2$ more times. 

This completes the construction. The total number of bicliques used is at most $n  + 2kn^{3/4} + k\sqrt{n}$ (from the padding step, the $\{1,2\}$-covering step, and the initial design, respectively). 
\end{proof}

\begin{remark} A key ingredient in the proof above is the $\{2k-2, 2k\}$-biclique covering of $K_n$ using $2k\sqrt{n}$ bicliques. It is shown in \cite{Cioaba2013} that $\sqrt{n/2}$ bicliques are necessary for this list covering, so the asymptotic dependence on $n$ cannot be decreased. \\
\end{remark}

\newpage

\section{Clique partitions of a graph and its complement}

\subsection{Improving the lower bound}\label{sec:lower_bd}

The construction in our proof of Theorem \ref{thm_lower} is based on the original construction in \cite{DEPW}, and the calculation of its clique partition number makes use of certain facts shown in \cite{DEPW} and \cite{PD}. Here we include the entire argument for the reader's convenience. Before proceeding with the construction, we need the following lemma, which has appeared in many places but perhaps first in Pullman and Donald \cite{PD}. Recall that the edge chromatic number $\chi'(G)$ of a graph $G$ is the minimum number of colors needed to color the edges of $G$ so that no two edges of the same color are incident to the same vertex. We use the notation $G \equiv H$ to denote the graph on vertices $V(G) \sqcup V(H)$ formed by adding all edges between $V(G)$ and $V(H)$.

\begin{lem}\label{chrom}
Let $G$ be any graph with $n$ vertices and $e$ edges. Then $\emph{\cp}(G \equiv \overline{K_\ell}) \geq n\ell - e$. If $\chi'(G) \leq \ell$, then $\emph{\cp}(G \equiv \overline{K_\ell}) = n\ell - e$. 
\end{lem}
\begin{proof}
Let $H = \overline{K_\ell}$ and let $E_{G - H}$ be the set of all $n\ell$ edges between $V(G)$ and $V(H)$. Suppose $C_1, \dots, C_r$ is a clique partition of $G \equiv H$. Since $C_i$ can have at most one vertex in $H$, it follows that $|E(C_i) \cap E(G)| \geq \binom{|E(C_i) \cap E_{G-H}| - 1}{2} \geq |E(C_i) \cap E_{G-H}| - 1$. Letting $S = \{i \, : \, E(C_i) \cap E_{G-H} \neq \emptyset\}$ and summing this inequality over $S$, we obtain
\be
e \geq \sum_{i \in S}|E(C_i) \cap E(G)| \geq \sum_{i \in S} |E(C_i) \cap E_{G-H}| - |S| \geq n\ell - r,
\ee
which implies $\cp(G \equiv H) \geq n\ell - e$. When $\chi'(G) \leq \ell$, we can assign each of the $\ell$ nodes in $H$ to one of the $\ell$ color classes of a valid edge coloring in $G$, and obtain a collection of triangles of the form $\{v, x, y\}$, for $v \in H$ and $(x,y) \in E(G)$ that has been given color $v$ in the edge coloring. No edge in $E_{G - H}$ will be used twice precisely because no vertex in $G$ is incident to two edges of the same color. This gives a collection of $e$ edge-disjoint triangles that cover all the edges in $G$, and leaves at most $n\ell - 2e$ edges left to cover. Adding in those remaining edges yields a clique partition of size at most $n\ell - e$.
\end{proof}

\tikzset{every picture/.style={line width=0.6pt}} 
\begin{figure}
    \centering

\begin{tikzpicture}[x=0.6pt,y=0.6pt,yscale=-1,xscale=1]
\draw  [fill={rgb, 255:red, 0; green, 0; blue, 0 }  ,fill opacity=0 ] (108,114) .. controls (108,102.4) and (117.4,93) .. (129,93) .. controls (140.6,93) and (150,102.4) .. (150,114) .. controls (150,125.6) and (140.6,135) .. (129,135) .. controls (117.4,135) and (108,125.6) .. (108,114) -- cycle ;
\draw  [fill={rgb, 255:red, 0; green, 0; blue, 0 }  ,fill opacity=0 ] (246,114) .. controls (246,102.4) and (255.4,93) .. (267,93) .. controls (278.6,93) and (288,102.4) .. (288,114) .. controls (288,125.6) and (278.6,135) .. (267,135) .. controls (255.4,135) and (246,125.6) .. (246,114) -- cycle ;
\draw  [fill={rgb, 255:red, 0; green, 0; blue, 0 }  ,fill opacity=0.58 ] (128,200) .. controls (128,188.4) and (137.4,179) .. (149,179) .. controls (160.6,179) and (170,188.4) .. (170,200) .. controls (170,211.6) and (160.6,221) .. (149,221) .. controls (137.4,221) and (128,211.6) .. (128,200) -- cycle ;
\draw  [fill={rgb, 255:red, 0; green, 0; blue, 0 }  ,fill opacity=0.58 ] (228,200) .. controls (228,188.4) and (237.4,179) .. (249,179) .. controls (260.6,179) and (270,188.4) .. (270,200) .. controls (270,211.6) and (260.6,221) .. (249,221) .. controls (237.4,221) and (228,211.6) .. (228,200) -- cycle ;
\draw   (174,57.5) .. controls (174,42.86) and (185.86,31) .. (200.5,31) .. controls (215.14,31) and (227,42.86) .. (227,57.5) .. controls (227,72.14) and (215.14,84) .. (200.5,84) .. controls (185.86,84) and (174,72.14) .. (174,57.5) -- cycle ;
\draw    (170,200) -- (228,200) ;
\draw    (170,206) -- (227,206) ;
\draw    (129,135) -- (144,180) ;
\draw    (261,134) -- (250,179) ;
\draw    (267,135) -- (256,181) ;
\draw    (135,134) -- (150,179) ;
\draw [fill={rgb, 255:red, 0; green, 0; blue, 0 }  ,fill opacity=1 ]   (227,57.5) -- (267,93) ;
\draw [fill={rgb, 255:red, 0; green, 0; blue, 0 }  ,fill opacity=1 ]   (227,65) -- (259,93) ;
\draw    (174,57.5) -- (135,94) ;
\draw    (174,66) -- (142,96) ;
\draw  [fill={rgb, 255:red, 0; green, 0; blue, 0 }  ,fill opacity=0 ] (367,112) .. controls (367,100.4) and (376.4,91) .. (388,91) .. controls (399.6,91) and (409,100.4) .. (409,112) .. controls (409,123.6) and (399.6,133) .. (388,133) .. controls (376.4,133) and (367,123.6) .. (367,112) -- cycle ;
\draw  [fill={rgb, 255:red, 0; green, 0; blue, 0 }  ,fill opacity=0 ] (505,112) .. controls (505,100.4) and (514.4,91) .. (526,91) .. controls (537.6,91) and (547,100.4) .. (547,112) .. controls (547,123.6) and (537.6,133) .. (526,133) .. controls (514.4,133) and (505,123.6) .. (505,112) -- cycle ;
\draw  [fill={rgb, 255:red, 0; green, 0; blue, 0 }  ,fill opacity=0.58 ] (387,198) .. controls (387,186.4) and (396.4,177) .. (408,177) .. controls (419.6,177) and (429,186.4) .. (429,198) .. controls (429,209.6) and (419.6,219) .. (408,219) .. controls (396.4,219) and (387,209.6) .. (387,198) -- cycle ;
\draw  [fill={rgb, 255:red, 0; green, 0; blue, 0 }  ,fill opacity=0.58 ] (487,198) .. controls (487,186.4) and (496.4,177) .. (508,177) .. controls (519.6,177) and (529,186.4) .. (529,198) .. controls (529,209.6) and (519.6,219) .. (508,219) .. controls (496.4,219) and (487,209.6) .. (487,198) -- cycle ;
\draw   (433,55.5) .. controls (433,40.86) and (444.86,29) .. (459.5,29) .. controls (474.14,29) and (486,40.86) .. (486,55.5) .. controls (486,70.14) and (474.14,82) .. (459.5,82) .. controls (444.86,82) and (433,70.14) .. (433,55.5) -- cycle ;
\draw    (429,198) -- (487,198) ;
\draw    (429,204) -- (486,204) ;
\draw    (388,133) -- (403,178) ;
\draw    (520,132) -- (509,177) ;
\draw    (526,133) -- (515,179) ;
\draw    (394,132) -- (409,177) ;
\draw    (486,55.5) -- (526,91) ;
\draw    (486,63) -- (518,91) ;
\draw    (433,55.5) -- (394,92) ;
\draw    (433,64) -- (401,94) ;

\draw (200.5,57.5) node  [rotate=-359.4]  {$G$};
\draw (459.5,55.5) node  [rotate=-359.4]  {$\overline{G}$};
\draw (149,200) node  [rotate=-359.4]  {$K_{\ell }$};
\draw (408,198) node  [rotate=-359.4]  {$K_{\ell }$};
\draw (249,200) node  [rotate=-359.4]  {$K_{\ell }$};
\draw (508,198) node  [rotate=-359.4]  {$K_{\ell }$};
\draw (129,114) node  [rotate=-359.4]  {$\overline{K_{\ell }}$};
\draw (267,114) node  [rotate=-359.4]  {$\overline{K_{\ell }}$};
\draw (388,112) node  [rotate=-359.4]  {$\overline{K_{\ell }}$};
\draw (526,112) node  [rotate=-359.4]  {$\overline{K_{\ell }}$};

\end{tikzpicture}
    \caption{The graph $H_\ell(G)$ and its complement.}
    \label{fig_graphs}
\end{figure}
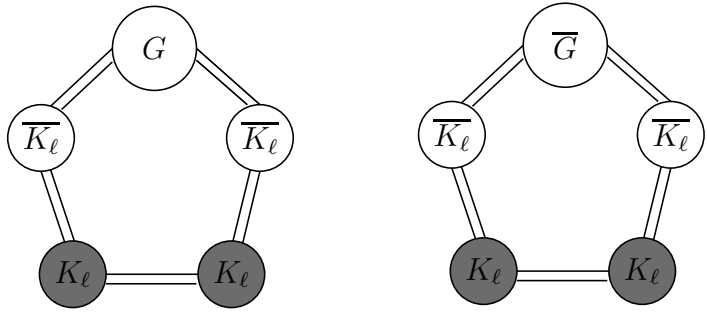

\textbf{The construction:} Let $\ell$ and $m$ be any positive integers, and let $G$ be any graph on $m$ vertices. We define $H_\ell = H_\ell(G)$ to be the graph in Figure \ref{fig_graphs}, where the double lines are to be interpreted in the same way as the $\equiv$ symbol, i.e. including all possible edges between the vertices on either end. Observe that $\overline{H_\ell(G)} \cong H_\ell(\overline{G})$, and that the edges of $H_\ell(G)$ can be split into $X_\ell(G) := G \equiv \overline{K_{2\ell}}$ and $Y_\ell = \overline{K_\ell} \equiv K_\ell \equiv K_\ell \equiv \overline{K_\ell}$, as depicted in Figure \ref{fig decomp}. Clearly $\chi'(G) \leq \chi'(K_m)$, which is at most $m$, since we can assign the numbers $0, 1, \dots, m-1$ to each vertex and color the edge $(i,j)$ by $i - j \mod m$. So if $m\leq 2\ell$, Lemma \ref{chrom} implies that $\cp(X_\ell(G)) = m\ell - e(G)$. Therefore 
\bea \nn
\cp(H_{\ell}(G)) + \cp(\overline{H_\ell(G)}) &=& \cp(H_{\ell}(G)) + \cp(H_\ell(\overline{G}))\\ \nn
&=& \cp(X_\ell(G)) + \cp(X_\ell(\overline{G})) + 2\cp(Y_\ell) \\\nn &=& 2m\ell - \binom{m}{2} + 2\cp(Y_\ell)
\eea
for \emph{any} graph $G$ on $m \leq 2\ell$ vertices. (In fact, this still gives a lower bound on $\cp(H_{\ell}(G)) + \cp(\overline{H_\ell(G)})$ for any $G$ and any $m$.) The term $\cp(Y_\ell)$ was computed in \cite{DEPW}, and we include this calculation in the Appendix:

\begin{figure}
    \centering
    
\begin{tikzpicture}[x=0.6pt,y=0.6pt,yscale=-1,xscale=1]

\draw  [fill={rgb, 255:red, 0; green, 0; blue, 0 }  ,fill opacity=0 ] (161,166.5) .. controls (161,148.55) and (175.55,134) .. (193.5,134) .. controls (211.45,134) and (226,148.55) .. (226,166.5) .. controls (226,184.45) and (211.45,199) .. (193.5,199) .. controls (175.55,199) and (161,184.45) .. (161,166.5) -- cycle ;
\draw   (167,85.5) .. controls (167,70.86) and (178.86,59) .. (193.5,59) .. controls (208.14,59) and (220,70.86) .. (220,85.5) .. controls (220,100.14) and (208.14,112) .. (193.5,112) .. controls (178.86,112) and (167,100.14) .. (167,85.5) -- cycle ;
\draw [fill={rgb, 255:red, 0; green, 0; blue, 0 }  ,fill opacity=1 ]   (199,111) -- (199,135) ;
\draw  [fill={rgb, 255:red, 0; green, 0; blue, 0 }  ,fill opacity=0 ] (369,99) .. controls (369,87.4) and (378.4,78) .. (390,78) .. controls (401.6,78) and (411,87.4) .. (411,99) .. controls (411,110.6) and (401.6,120) .. (390,120) .. controls (378.4,120) and (369,110.6) .. (369,99) -- cycle ;
\draw  [fill={rgb, 255:red, 0; green, 0; blue, 0 }  ,fill opacity=0 ] (468,100) .. controls (468,88.4) and (477.4,79) .. (489,79) .. controls (500.6,79) and (510,88.4) .. (510,100) .. controls (510,111.6) and (500.6,121) .. (489,121) .. controls (477.4,121) and (468,111.6) .. (468,100) -- cycle ;
\draw  [fill={rgb, 255:red, 0; green, 0; blue, 0 }  ,fill opacity=0.58 ] (369,176) .. controls (369,164.4) and (378.4,155) .. (390,155) .. controls (401.6,155) and (411,164.4) .. (411,176) .. controls (411,187.6) and (401.6,197) .. (390,197) .. controls (378.4,197) and (369,187.6) .. (369,176) -- cycle ;
\draw  [fill={rgb, 255:red, 0; green, 0; blue, 0 }  ,fill opacity=0.58 ] (469,176) .. controls (469,164.4) and (478.4,155) .. (490,155) .. controls (501.6,155) and (511,164.4) .. (511,176) .. controls (511,187.6) and (501.6,197) .. (490,197) .. controls (478.4,197) and (469,187.6) .. (469,176) -- cycle ;
\draw    (411,176) -- (469,176) ;
\draw    (411,182) -- (468,182) ;
\draw    (393,119) -- (393,157) ;
\draw    (387,120) -- (387,157) ;
\draw    (494,121) -- (494,157) ;
\draw    (488,120) -- (488,157) ;
\draw [fill={rgb, 255:red, 0; green, 0; blue, 0 }  ,fill opacity=1 ]   (189,111) -- (189,135) ;

\draw (193.5,85.5) node  [rotate=-359.4]  {$G$};
\draw (193.5,166.5) node  [rotate=-359.4]  {$\overline{K_{2\ell } \ }$};
\draw (390,176) node  [rotate=-359.4]  {$K_{\ell }$};
\draw (490,176) node  [rotate=-359.4]  {$K_{\ell }$};
\draw (390,99) node  [rotate=-359.4]  {$\overline{K_{\ell }}$};
\draw (489,100) node  [rotate=-359.4]  {$\overline{K_{\ell }}$};

\end{tikzpicture}

    \caption{Decomposing $H_\ell$ into the edge-disjoint union of the two graphs $X_\ell(G) = G \equiv \overline{K_{2\ell}}$ (left) and $Y_\ell = \overline{K_\ell} \equiv K_\ell \equiv K_\ell \equiv \overline{K_\ell}$ (right).}
    \label{fig decomp}
\end{figure}
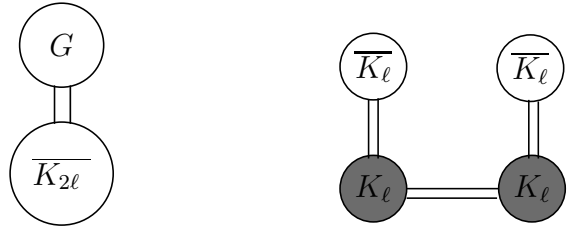

\begin{lem}[Lem. 2 and 3 in \cite{DEPW}]\label{cp(Y)}
For any $\ell$, $\emph{\cp}(Y_\ell) \geq \frac{7}{4}\ell^2 + O(\ell)$, and this is tight infinitely often. 
\end{lem}

So for any $G$ on $m$ vertices, we have \be \label{m,l bound}
\cp(H_{\ell}(G)) + \cp(\overline{H_\ell(G)}) \geq 2m\ell - \binom{m}{2} + \frac{7}{2}\ell^2.
\ee
Note that $H_\ell(G)$ has $n := m + 4\ell$ vertices, so when we maximize (\ref{m,l bound}) in $m$ while keeping $n$ fixed, we find that the optimum occurs at $m = \frac{9}{8}\ell$. At this value of $m$, the lower bound is $(8 - \frac{81}{128})\ell^2 + O(\ell)$ for a graph on $\frac{41}{8}\ell$ vertices, implying that, for infinitely many $n$, \be \nn 
\max_{G \in \mathcal{G}_n} \cp(G) + \cp(\overline{G}) \geq \frac{(8 - \frac{81}{128})}{(\frac{41}{8})^2} n^2 + O(n) = \frac{23}{82} n^2 + O(n).
\ee
Note that if $G$ is a self-complementary graph (i.e. $G \cong \overline{G}$), then $H_{\ell}(G)$ is also self-complementary. \qed

\subsection{Improving the upper bound}\label{sec:upper_bd}

The problem of partitioning a graph $G$ into as few cliques as possible is equivalent to the problem of \emph{packing} disjoint copies of $K_3, K_4, \dots,  K_n$ inside of $G$ in such a way as to maximize a certain linear objective function. Indeed, given a clique partition $C$ of $G$, let $C_i$ denote the number of cliques of size $i$ in $C$, for $i = 2, \dots, n$. Then
$|C| = \sum_{i=2}^{n} C_i$ and $\sum_{i=2}^n \binom{i}{2}C_i = |E(G)|$, so
\bea \nn
\cp(G) &=& \min_C |C| \\
&=& E(G) - \underbrace{\max_C \sum_{i \geq 3} \left(\binom{i}{2} - 1\right)C_i}_{=: v(G)}.
\eea
We will also consider \emph{$r$-restricted clique packings/partitions}, in which the largest clique can have size at most $r$. We define $\cp(G, r)$ to be the minimum number of cliques of size at most $r$ needed to partition the edges of $G$. Equivalently, $\cp(G, r) = E(G) - v_r(G)$, where
\be\label{v_r def}
v_r(G) := \max_C \sum_{i = 3}^r \left(\binom{i}{2} - 1\right)C_i.
\ee
Clearly $\cp(G, r) \geq \cp(G)$, and one would expect the numbers $\cp(G,r)$ and $\cp(G)$ to be relatively close for large $r$. This is indeed the case, as we show in the following lemma.

\begin{lem}\label{r close}
For any $\epsilon > 0$, there exists an integer $r_0 = r_0(\epsilon)$ such that for any $r \geq r_0$ and any graph $G$ on $n$ vertices, 
$$\emph{\cp}(G, r) \leq \emph{\cp}(G) + \epsilon \cdot n^2.$$
\end{lem}

\begin{proof}
We make use of the following fact, which is a straightforward consequence of Wilson's theorem \cite{Wilson}: for any fixed $t \geq 2$ and $\epsilon > 0$, there is an integer $m_0 = m_0(t, \epsilon)$ such that for all $m \geq m_0$,  there is a partition of $K_m$ into edge-disjoint copies of $K_t$ and at most $\epsilon m^2$ leftover edges. Set $t = \frac{1}{2\epsilon}$ and $r_0 = m_0(t, \epsilon/5)$.

Let $C$ be a clique partition with $|C| = \cp(G)$. For any $r \geq r_0$, we can obtain an $r$-restricted clique partition $\widetilde{C}$ from $C$ as follows:  keep each clique of size at most $r$, and, for each clique $K_m$ with $m > r$, decompose it into at most $\binom{m}{2}/\binom{t}{2}$ copies of $K_t$ and cover the remaining edges (of which there are at most $\frac{\epsilon}{5} \cdot m^2$) with $K_2$'s. This gives a clique partition $\widetilde{C}$ of size
\bea \nn
|\widetilde{C}| &\leq& \sum_{i = 2}^{r} C_i + \sum_{i > r} \left(\frac{\binom{i}{2}}{\binom{t}{2}} + \frac{\epsilon}{5} i^2  \right)C_i \\ \nn
& \leq & \sum_{i=2}^n C_i + \epsilon \cdot \sum_{i=2}^n \binom{i}{2} C_i \\ \nn
& = & |C| + \epsilon \cdot |E(G)|
\eea
from which the lemma follows.
\end{proof}

\subsubsection{Fractional clique packings}
For a fixed family $\mathcal{F}$ of graphs and any graph $G$, let $\binom{G}{\mathcal{F}}$ denote the set of (unlabeled, non-induced) subgraphs of $G$ which are isomorphic to some $F \in \mathcal{F}$. Following Keevash and Sudakov \cite{KS} and Yuster \cite{Y1}, we say a function $\psi: \binom{G}{\mathcal{F}} \to [0,1]$ is a \emph{fractional $\mathcal{F}$-packing of $G$} if for every edge $e \in E(G)$, we have
$$\sum_{e \in H \in \binom{G}{\mathcal{F}}} \psi(H) \leq 1.$$
We denote by $G_{\mathcal{F}}$ the polyhedron of all fractional $\mathcal{F}$-packings of $G$. As we are interested in the fractional analogue of clique packings, we will only be concerned with families of the form
$$\mathcal{F}_r := \{K_3, K_4, \dots, K_r\}.$$
Let $\nu_r(G)$ be the value of the linear program
\be\label{LP_cp}
\max_{\psi \in G_{\mathcal{F}_r}} \sum_{H \in \binom{G}{\mathcal{F}_r}} \left(\binom{|H|}{2} -1\right)\psi(H).
\ee
When the objective function is simply $\sum_{H \in \binom{G}{\mathcal{F}}}\psi(H)$, and the family $\mathcal{F} = \{F\}$ is just a single graph, a theorem of Haxell and R{\"o}dl \cite{HR} implies that relaxing the domain of maximization from (integer) packings to fractional packings can only change the value of the optimum by $o(n^2)$.  Subsequently, Yuster \cite{Y1} extended this result to arbitrary families of graphs. For finite families (such as $\mathcal{F}_r$), Yuster's proof easily extends to arbitrary linear objective functions \cite{Y2}. Therefore:
\begin{thm}\label{YHR} For any $r \geq 3$ and $G \in \mathcal{G}_n$,
$$v_r(G) - \nu_r(G) = o(n^2).$$
\end{thm}
The advantages of studying fractional clique packings rather than clique partitions are twofold. First, solving the linear program (\ref{LP_cp}) is computationally feasible, unlike the corresponding integer program. Second, they can be \emph{averaged}, which not only enables one to turn finite computations into asymptotic bounds, but also allows one to leverage the results of a search on $n$ vertices to reduce the search space when looking for a minimizer on $n+1$ vertices. This is the approach used by Keevash and Sudakov in \cite{KS}, and the following averaging lemma (for a different LP) appears as their Lemma 2.1, with the same proof.

For each $r$, define
$$f_r(n) := \min_{G \in \mathcal{G}_n} \nu_r(G) + \nu_r(\overline{G}).$$

\begin{lem}\label{averaging}
For any $r \geq 3$, the sequence $\frac{f_r(n)}{n(n-1)}$ is increasing in $n$. 
\end{lem}

\begin{proof}
Let $G \in \mathcal{G}_{n+1}$, and let $G_1, \dots, G_{n+1}$ be the induced subgraphs on the vertex subsets of size $n$. Let $\psi_i, \overline{\psi_i}$ be optimal fractional packings on $G_i$ and $\overline{G_i}$. Since each edge of $G$ (and $\overline{G}$) occurs in $n-1$ of the $G_i$, we have that $$\psi := \frac{1}{n-1}\sum_{i=1}^{n+1} \psi_i, \hspace{20pt} \overline{\psi} := \frac{1}{n-1}\sum_{i=1}^{n+1}\overline{\psi_i}$$
are fractional packings on $G$ and $\overline{G}$ with combined objective value of at least $\frac{n+1}{n-1}f_r(n)$, and hence $\frac{f_r(n+1)}{(n+1)n} \geq \frac{(n+1)f_r(n)}{n(n-1)(n+1)} = \frac{f_r(n)}{n(n-1)}$, as claimed.
\end{proof}

Since the sequence $\frac{f_r(n)}{n(n-1)}$ is obviously bounded above by $1/2$, it follows that it converges to a limit $c_r \in (0, 1/2)$. Since $c_r$ is increasing in $r$, the sequence $\{c_r\}$ also converges to a limit that we will call $c_{\infty}$.

\begin{thm}
$$\max_{G \in \mathcal{G}_n} \emph{\cp}(G) + \emph{\cp}(\overline{G}) \sim \left(\frac{1}{2}-c_{\infty}\right) n^2.$$
\end{thm}

\begin{proof}
This essentially follows from Lemma \ref{r close} and Theorem \ref{YHR}. More explicitly, for any $\epsilon > 0$, we can pick $r$ large enough so that $|\cp(G) - \cp(G,r)| < \epsilon n^2$ for any $G \in \mathcal{G}_n$, and $|c_r - c_{\infty}| < \epsilon$. Now pick $n$ large enough so that $|v_r(G) - \nu_r(G)| < \epsilon  n^2$ for any $G \in \mathcal{G}_n$ and $|f_r(n) - c_r  n^2| < \epsilon n^2.$ It follows that 
$$\max_{G \in \mathcal{G}_n} \cp(G) + \cp(\overline{G}) \in \left(\frac{1}{2}-c_{\infty} \pm 8\epsilon\right) n^2$$
for $n$ sufficiently large.
\end{proof}

The same argument shows that $\max_{G \in \mathcal{G}_n} {\cp}(G, r) + {\cp}(\overline{G},r) \sim \left(\frac{1}{2}-c_{r}\right) n^2$. Let us define $\alpha_r := \frac{1}{2} - c_r$, and $\alpha_{\infty} = \frac{1}{2} - c_\infty$. We seek an upper bound on $\alpha_{\infty}$, and since $\alpha_\infty \leq \alpha_r = \frac{1}{2}-c_r \leq \frac{1}{2} - \frac{f_r(n)}{n(n-1)}$ for any $n$, it suffices for our purposes to compute a lower bound on the value of $\frac{f_r(n)}{n(n-1)}$ for any particular pair of positive integers $(r, n)$. For example, a modern computer can compute $f_4(8) = 6$ numerically by solving the LP (\ref{LP_cp}) on every non-isomorphic graph on $8$ vertices. This shows that $\alpha_\infty \leq \alpha_4 \leq \frac{1}{2} -  \frac{6}{8\cdot 7} = \frac{11}{28} \approx 0.3928$. This already beats the best bound one can get from purely Ramsey-based arguments\footnote{As was remarked in \cite{DEPW}, one can begin with a maximal collection of edge disjoint $K_r$'s (instead of triangles) in $G$ and $\overline{G}$, and bound the number of remaining edges (using Turan's theorem) by $\xi_r n^2$, where $\xi_r := \frac{1}{2}-\frac{1}{2R(r,r) - 2}$, and the iterate on the remaining edges with cliques of size $K_{r-1}$, etc. It is not hard to see that the bound one obtains is 
\be\nn 
\cp(G) + \cp(\overline{G}) \leq \left(\xi_3 + \frac{\xi_4 - \xi_3}{3} + \frac{\xi_5 - \xi_4}{6} + \cdots + \frac{\xi_r - \xi_{r-1}}{\binom{r-1}{2}} + \frac{\frac{1}{2} - \xi_r}{\binom{r}{2}}\right) n^2.
\ee
Even using the most optimistic (i.e. smallest) of the possible values for $R(k,k)$ for $k \geq 5$, this approach will not yield an upper bound better than $0.41n^2$.}, although it does not beat the Keevash-Sudakov triangle packing bound. In the remainder of this section, we improve this bound in two ways: first, we show in Section \ref{sec:ramsey}, we can \emph{combine} Ramsey-type arguments with estimates on $f_r(n_0)$ to yield better estimates on $f_r(n)$ for $n$ much larger than $n_0$; second, in Section \ref{sec:computation} we compute the exact value of $f_4(n)$ up to $n = 19$, using an algorithm of Keevash and Sudakov that is significantly more efficient than brute force search.

\subsubsection{Ramsey-type improvements}\label{sec:ramsey}

In \cite{KS}, it was observed that the averaging argument in Lemma \ref{averaging} can be improved, in a sense, by using a different decomposition of $G$ into smaller subgraphs based on a greedy packing as described in the introduction. In particular, given any bicoloring of $K_{3n}$, greedily select vertex-disjoint monochromatic triangles $T_1, \dots, T_{i}$. The fact that $R(3,3) = 6$ guarantees that we can do this until $3$ vertices remain, giving us $n-1$ triangles $T_1, \dots, T_{n-1}$, and one set of 3 vertices denoted $T_n$. Consider the $3^n$ colorings $c$ of $K_n$ obtained by picking one vertex in each $T_i$ and the edges between them. Each coloring has some fractional packing $\psi_c$ of weight at least $f_3(n)$, and since each edge between $T_i$ and $T_j$ for $i \neq j$ occurs in exactly $3^{n-2}$ of these, the average $3^{-(n-2)}\sum_c \psi_c$ is a valid fractional packing in $K_{3n}$ of weight at least $9f_3(n)$. Since each of the monochromatic triangles $T_1, \dots, T_{n-1}$ are edge disjoint from this packing, they can be included as well, yielding a lower bound
\be \label{3,3} f_3(3n) \geq 9f_3(n) + 2(n-1).\ee
Since $R(4,4) = 18$, we can greedily find vertex disjoint monochromatic copies of $K_4$, $H_1, \dots, H_{n-4}$, with 16 vertices remaining. From the remaining vertices, we can find edge disjoint monochromatic triangles $T_{n-3}, T_{n-2}, T_{n-1}, T_{n}$, which we join with the remaining four vertices to form $H_{n-3}, \dots, H_n$, each of size four. Repeating the same process as above, we see that 
\be\label{4,4}
f_4(4n) \geq 16f_4(n) + 5(n-4) + 8.
\ee
For $r = 5$, we can use the bound $R(5,5) \leq 48$ to find $n-9$ vertex-disjoint copies of $K_5$, with 45 vertices left over. We can then find $\lceil(45-18)/4\rceil = 7$ copies of $K_4$, with 17 vertices left over, in which we can find 2 monochromatic triangles, and distribute the remaining vertices so that each of these 11 parts has size 5. Arguing as above, this then implies
\be\label{5,5}
f_5(5n) \geq 25f_5(n) + 9(n-9) + 37.
\ee
We omit the details, but using similar arguments and the Ramsey number bounds $R(6,6) \leq 165$ and $R(7,7) \leq 540$ yields the inequalities 
\bea
\label{6,6}f_6(6n) &\geq& 36f_6(n) + 14n - 151\\
\label{7,7}f_7(7n) &\geq& 49f_7(n) + 20n - 532.
\eea
According to Andr\'{a}s Gy\'{a}rf\'{a}s \cite{Gyarfas}, Paul Erd\H{o}s, sitting in the Atlanta Airport in 1995, asked his companions whether every bicoloring of the edges of $K_{R(k,k)}$ contains \emph{two edge-disjoint monochromatic copies of} $K_k$. Ralph Faudree pointed out that this is not true, at which point Erd\H{o}s asked for the smallest number $n(k)$ for which any bicoloring of $K_{n(k)}$ \emph{does} contain two edge-disjoint monochromatic $K_k$'s. The next day, Faudree showed $n(3) = 7$, and some time later, Gy\'{a}rf\'{a}s showed $n(4) = 19$. For our purposes, however, we require \emph{vertex-disjoint} monochromatic copies of $K_r$. In the appendix we give an argument, inspired by the proof of $n(4) = 19$ by Gy\'{a}rf\'{a}s, showing that $n = 20$ is sufficient to find two vertex-disjoint monochromatic $K_4$'s, provided there is also a monochromatic $K_5$:
\begin{lem}\label{fruit_salad}
Any bicoloring of the edges of $K_{20}$ with a monochromatic copy of $K_5$ contains two vertex-disjoint monochromatic copies of $K_4$.
\end{lem}
With this lemma in hand, we can obtain a slight improvement over (\ref{4,4}):

\begin{lem}\label{4,4 improved}
For any $n \geq 12$, $f_4(4n) \geq 16f_4(n) + 5n - 9.$
\end{lem}
\begin{proof}
Consider any bicoloring of $K_{4n}$. Since $4n \geq 48 \geq R(5,5)$, there is some monochromatic copy of $K_5$ -- call this subgraph $N$. While there are at least $R(4,4) = 18$ vertices in $K_{4n} \setminus N$, we can greedily select vertex-disjoint monochromatic copies of $K_4$ in $K_{4n}\setminus N$, $H_1, \dots, H_{n-5}$. This leaves a set $S$ of 15 remaining vertices. By Lemma \ref{fruit_salad}, the coloring induced on $S \cup N \cong K_{20}$ has two vertex disjoint copies of $K_4$, which we call $H_{n-4}$ and $H_{n-3}$. Removing the vertices in $H_{n-4} \cup H_{n-3}$ from $N \cup S$, we are left with 12 vertices, which must contain three vertex-disjoint monochromatic triangles $T_1, T_2$ and $T_3$. This leaves behind a set of three vertices $\{v_1, v_2, v_3\}$. Decomposing $K_{4n}$ into the $n$ blocks of size 4
\be \nn
H_1, \dots, H_{n-3}, T_1\cup\{v_1\}, T_2\cup\{v_2\}, T_3\cup\{v_3\},
\ee
we consider the $4^n$ edge-colorings $c$ of $K_n$ obtained by picking one vertex from each part. Each of these has a fractional clique packing $\psi_c$ of size at least $f_4(n)$, and since each edge is used in $4^{-(n-2)}$ such $\psi_c$, we know that $4^{-(n-2)}\sum_c \psi_c$ is a valid packing in $K_{4n}$. Adding the copies of $K_4$ and $K_3$ inside the $n$ individual blocks, we see that $f_4(4n) \geq 16f_4(n) + 5(n-3) + 6$.
\end{proof}

\subsubsection{Computer-aided calculations}\label{sec:computation}

We next describe a generalization of the algorithm used by Keevash and Sudakov in the case of triangle packings \cite{KS}, which we call the KS extension method. For any finite family of graphs $\mathcal{F} = \{H_1, \dots, H_r\}$, any graph $G \in \mathcal{G}_n$, and any vector $\Gamma \in \R^{\mathcal{F}}$, we let $\nu_{\mathcal{F}, \Gamma}(G)$ be the value of the linear program
\be\label{F-LP}
\max_{\psi \in G_{\mathcal{F}}} \sum_{H \in \binom{G}{\mathcal{F}}}\Gamma(H) \psi(H)
\ee
and define $\Lambda(\mathcal{F}, \Gamma, n) := \min_{G \in \mathcal{G}_n} \nu_{\mathcal{F}, \Gamma}(G) + \nu_{\mathcal{F}, \Gamma}(\overline{G})$.
For any $\ell \in \R$, and any set $L$ of graphs, define $$\mathcal{L}(L, \ell) := \{G \in L \, : \,\nu_{\mathcal{F}, \Gamma}(G) + \nu_{\mathcal{F}, \Gamma}(\overline{G}) \leq \ell\},$$ and let $\Lambda_{\mathcal{F}, \Gamma}(L) = \min_{G \in L}\nu_{\mathcal{F}, \Gamma}(G') + \nu_{\mathcal{F}, \Gamma}(\overline{G'})$. We also define $\ext(L)$ be the set of \emph{one-vertex extensions} of the graphs in $L$.
The KS extension method is based on the following observation: by Lemma \ref{averaging}, any graph $G' \in \mathcal{G}_{n+1}$ with $ \nu_{\mathcal{F}, \Gamma}(G') + \nu_{\mathcal{F}, \Gamma}(\overline{G'}) \leq \frac{n+1}{n-1} \cdot \ell$ must be a one-vertex extension of some graph in $\mathcal{L}(\mathcal{G}_n, \ell)$. In other words, if $\{\ell_n\}_{n \in \N}$ is any sequence of numbers satisfying $\ell_{n+1} \geq \frac{n+1}{n-1}\ell_n$, then $$\mathcal{L}(\mathcal{G}_{n+1}, \ell_{n+1}) \subseteq \ext(\mathcal{L}(\mathcal{G}_n, \ell_n)).$$ Let us refer to such sequences $\ell_n$ as \emph{level sequences}. \\


\begin{algorithm}[H]\label{KS_boot}
\SetAlgoLined

  $n \leftarrow n_0$ \\
 compute $L = \mathcal{L}(\mathcal{G}_{n}, \ell_{n})$ (e.g. via exhaustive search)\\

 \While{$L \neq \emptyset$}{
  $\Lambda[n] = \Lambda_{\mathcal{F}, \Gamma}(L)$\\ 
 $S \leftarrow \ext(L)$\\
  $L \leftarrow \mathcal{L}(S, \ell_{n+1})$\\
  $\Lambda[n+1] = \ell_{n+1}$\\
  $n \leftarrow n + 1$
 }
 \textbf{return} $\Lambda$
 \caption{KS Extension Method}
\end{algorithm}

\vspace{10pt}
Note that the sequence $\ell_n$ used by Algorithm \ref{KS_boot} does not have to be determined before runtime; as long as it is guaranteed to be a level sequence, this guarantees the loop invariant $\mathcal{L}(\mathcal{G}_n, \ell_n) \subseteq L$, and hence $\Lambda[n] \leq \Lambda_{\mathcal{F}, \Gamma}(\mathcal{G}_n)$. In \cite{KS}, they choose a parameter $d$ (called the ``search depth''), and define $\ell_n$ recursively by taking $\ell_{n_0} = +\infty$ and
$\ell_{n+1}$ to be $\frac{n+1}{n-1} \cdot \alpha_n$, where $\alpha_n$ is either (a) the $d$th smallest value in the set $\{\nu_{\mathcal{F}, \Gamma}(G') + \nu_{\mathcal{F}, \Gamma}(\overline{G'}) \, : \, G \in \mathcal{L}(\mathcal{G}_n, \ell_n)\}$, if this set has at least $d$ elements, or (b) $\ell_n$, if the set has fewer than $d$ elements. The role of $d$ is to limit the number of graphs stored in the set $L$. If $d = \infty$, then Algorithm \ref{KS_boot} has to solve the LP (\ref{F-LP}) on every graph up to size $n$ in order to compute $\Lambda_{\mathcal{F}, \Gamma}(\mathcal{G}_n)$, while if $d$ is too small, then the while loop will terminate after a small number of iterations. 
\begin{table}{\centering
\resizebox{1\textwidth}{!}{

\begin{tabular}{|c|c|c|c|c|c|c|c|c|c|c|c|c|c|c|c|}
\hline
 \diagbox[]{$i$}{$n$}&6&7&8&9&10&11&12&13&14&15&16&17&18&19&20   \\ \hline
1 &2 &4 &6&        8& 11& 15 &19 &23 &27 &33 &39 &45 &51 &57 & $>$ 64.725  \\ \hline
2 &4 &5 &7&        9& 12& 16 &20  &24 &28 &34 &40 &46 &52 &58 & *         \\ \hline
3 &5 &6 &8&        10& 12.5&16.5&20.5  &24.5 &28.5 &34.5 &40.5 &* &* &* &*    \\ \hline
4 &6 &7 &9&        11& 13& 17 &21  &25 &29 &34.75 &40.75 &* &* &* &*   \\ \hline
5 &7 &8 &9.5&       12& 14& 17.5 &21.25  &25.25 &29.25 &35 &* &* &* &* &*   \\ \hline
6 &8 &$8.\overline{3}$ &10&  $12.\overline{3}$& 14.5&18 &21.5  &25.5 &29.5 &35.25 &* &* &* &* &*  \\ \hline
7 &$8.\overline{3}$ &8.5&$10.\overline{3}$&12.5&$14.\overline{6}$& 18.25 &22  &26 &30 &35.5 &* &* &* &* &*  \\ \hline
8 &9 &9 &10.5&     12.6& 14.8& $18.\overline{3}$ &22.25  &26.25 &30.25 &* &* &* &* &* &*  \\ \hline
9 &$10.\overline{6}$ &9.5& 10.6&$12.\overline{6}$& 15&  18.5& $22.\overline{3}$  &$26.\overline{3}$ &30.5 &* &* &* &* &* &*  \\ \hline
10 &12.5 &10&$10.\overline{6}$& 12.8& 15.5& $18.\overline{6}$& 22.5 &26.5 &30.75 &* &* &* &* &* &*  \\ \hline
11 &* &$10.\overline{3}$ &10.8  &   13& 15.6& 18.75&*  &* &31 &* &* &* &* &* &*  \\ \hline

\end{tabular}}}
\caption{The lowest values of $\nu_4(G) + \nu_4(\overline{G})$ for $G \in \mathcal{G}_n$, $n = 6, \dots, 19$, as found by the KS extension method. The level $\ell_{20}$ was $64.72527+$ when the algorithm terminated, which implies that $f_4(20) > 64.725$.}
\label{data}
\end{table}
We ran an implementation\footnote{There are other implementation details omitted from our pseudocode description of Algorithm \ref{KS_boot} that also have significant impact on its runtime and memory usage, such as how and when to prune isomorphisms, which LP solver to use, which value of $n_0$ to exhaust from, and how to split work among processors. Our implementation is similar to the one used in \cite{KS}, and we recommend reading their magma code, which can be found online at \url{https://people.math.ethz.ch/~ sudakovb/triangles-program}.} of this method on a 24-core computing grid with $d = 11$, starting with an exhaustive search on $n_0 = 6$ vertices, and obtained the results summarized in Table \ref{data}. The last column in particular implies $f_4(20) > 64.725$, which implies $c_4 > 0.1703$. Using Lemma \ref{4,4 improved}, and inequalities (\ref{5,5}), (\ref{6,6}), and (\ref{7,7}) (in that order), we can obtain the bound $c_7 \geq 0.1814$, which implies \be \nn
\max_{G \in \mathcal{G}_n} \cp(G) + \cp(\overline{G}) < 0.3186 n^2 + o(n^2).
\ee\qed

\makeatletter
\renewcommand\subsubsection{\@startsection{subsubsection}{3}{\z@}%
                                     {-3.25ex\@plus -1ex \@minus -.2ex}%
                                     {-1.5ex \@plus -.2ex}
                                     {\normalfont\normalsize\bfseries}}
\makeatother

\subsubsection*{Acknowledgments} 
D. Rohatgi and J. Wellens are indebted to Asaf Ferber and Vishesh Jain for introducing them to Conjecture \ref{deCaen conj}, and to the MIT Summer Program for Undergraduate Research for support during the summer of 2018, when parts of this research were done. The work of J. Urschel was supported in part by ONR Research Contract N00014-17-1-2177. The authors are grateful to Louisa Thomas for greatly improving the style of presentation.

\section*{Appendix}

\subsection*{Proof of Lemma \ref{cp(Y)}}
\theoremstyle{theorem}
\newtheorem*{cp(Y)}{Lemma \ref{cp(Y)}}
\begin{cp(Y)}
For any $\ell$, $\emph{\cp}(Y_\ell) \geq \frac{7}{4}\ell^2 + O(\ell)$, and this is tight infinitely often.
\end{cp(Y)}

\begin{proof}
Let $\mathcal{C} = \{C_1, \dots, C_k\}$ be an optimal clique partition of $Y_\ell$. Let us denote the left (according to Figure \ref{fig decomp}) copy of $K_\ell$ in $Y_\ell$ by $A$ and the right copy by $B$. Suppose that $\mathcal{C}' = \{C_1, \dots, C_t\}$, for some $t \leq k$, is the sub-collection of cliques which contain vertices in both $A$ and $B$. Let $E_A$ and $E_B$ be the edges in $A \cap \mathcal{C}'$ and $B \cap \mathcal{C}'$, so that $Y_\ell$ is the edge disjoint union of $(A\setminus E_A) \equiv \overline{K}_\ell$ and $(B\setminus E_B) \equiv \overline{K}_\ell$ with $\mathcal{C}'$, and therefore 
\be
\cp(Y_\ell) \geq 2\ell^2 - 2\binom{\ell}{2} + |E_A| + |E_B| + t .
\ee
If clique $C_i$ has $a_i$ vertices in $A$ and $b_i$ vertices in $B$, then \be\label{ab_const} \sum_i^t a_ib_i = \ell^2\ee and \be\label{ab_contri} |E_A| + |E_B| + t = \sum_{i=1}^t \left(\binom{a_i}{2} + \binom{b_i}{2} + 1\right).\ee
Minimizing (\ref{ab_contri}) over positive integers $a_i, b_i$ subject to the constraint (\ref{ab_const}), we see the minimum occurs when $a_i = b_i = 2$, i.e. each $C_i \in \mathcal{C}'$ is a $K_4$ with two vertices in each of $A$ and $B$. Therefore, at the minimum, $t = \ell^2/4$ and $|E_A| + |E_B| + t = 3\ell^2/4$, which gives $$\cp(Y_{\ell}) \geq 2\ell^2 - 2\binom{\ell}{2} + 3\ell^2/4 = \frac{7}{4}\ell^2 + O(\ell),$$
as claimed. Tightness follows from Theorem 4 in \cite{PSW}, which essentially guarantees the existence of a decomposition of the edges between $A$ and $B$ into disjoint $K_4$'s, whenever $\ell \geq 14$ is even. 
\end{proof}

\subsection*{Proof of Lemma \ref{fruit_salad}}

\theoremstyle{theorem}
\newtheorem*{fruit_salad}{Lemma \ref{fruit_salad}}
\begin{fruit_salad}
Any bicoloring of the edges of $K_{20}$ with a monochromatic copy of $K_5$ contains two vertex-disjoint monochromatic copies of $K_4$.
\end{fruit_salad}

\begin{proof}
 Suppose that we have a bicoloring of $K_{20}$ with a red copy $N =\{n_1,...,n_5\}$ of $K_5$. If there is a blue copy of $K_4$, then we are finished, because this blue copy and $N$ cannot share an edge, and therefore share at most one vertex. We may now assume that all monochromatic copies of $K_4$ are red.

We can address the case in which there exists a vertex $v$ such that it is incident to at least nine red and blue edges each relatively quickly. We denote by $R$ and $B$ the cliques on the red and blue neighbors of $v$, respectively. Because the Ramsey number $R(3,4) = 9$ and our graph has no blue copy of $K_4$, $R$ must contain a red copy of $K_3$. Moreover, $B$ cannot contain a blue copy of $K_3$, so $B$ must contain a red copy of $K_4$. Adding $v$ to the red copy of $K_3$ in $R$ results in two vertex-disjoint red copies of $K_4$, one in $R\cup v$ and one in $B$. We may now assume that all vertices have at least eleven incident edges of the same color.

Consider the case in which some vertex $v$ has two red and two blue edges adjacent to a red copy $M$ of $K_4$. If $v$ has at least eleven red edges, then it has at least nine red edges connected to $K_{20} \backslash (M \cup \{v\})$, which, by the same argument as above, implies $K_{20} \backslash M$ has a red copy of $K_4$. The same argument holds if $v$ has at least eleven blue edges. We may now assume that no vertex has two red and two blue edges adjacent to a red copy of $K_4$.

From here, we consider two cases:
\begin{itemize}
    \item[Case I:] Suppose that there exists five vertices $V = \{v_1,...,v_5\} \subset K_{20}\backslash N$, each with at least three red edges adjacent to $N$. Because no vertex has both two red and two blue edges adjacent to a red copy of $K_4$, each vertex of $V$ has at least four red edges adjacent to $N$. In addition, because our graph has no blue copy of $K_4$ every set $V \backslash v_i$ has a red edge. 
    
    Suppose that some vertex of $V$, without loss of generality called $v_1$, has five red edges adjacent to $N$. Without loss of generality, $\{v_2,v_3\}$ is a red edge in $V \backslash v_1$. There are at most two blue edges from $v_2$ or $v_3$ to $N$; without loss of generality assume they are not incident to $n_4$ or $n_5$. Then the subsets $\{v_1,n_1,n_2,n_3\}$ and $\{v_2,v_3,n_4,n_5\}$ are both red copies of $K_4$. So we may now assume that each vertex in $V$ has exactly four red edges adjacent to $N$.
    
    Let $f(v_i)$ denote the unique vertex in $N$ for which edge $\{v_i,f(v_i)\}$ is blue, and $f(V)$ denote the range of $f$. We consider several sub-cases, depending on the size of $|f(V)|$.
    
    Suppose $|f(V)|>2$. Let $\{v_1,v_2\}$ be (without loss of generality) a red edge in $V$. Both $v_1$ and $v_2$ have four red edges to $N$, so there are at least three vertices in $N$ (without loss of generality $n_1$, $n_2$, and $n_3$) such that $\{v_i,n_j\}$ is a red edge for all $i \in \{1,2\}$ and $j \in \{1,2,3\}$. By pigeonhole, $| f(V) \cap \{n_1,n_2,n_3\}|>0$, so (without loss of generality) suppose that $f(v_3) = n_1$. Then $\{v_1,v_2,n_1,n_2\}$ and $\{v_3,n_3,n_4,n_5 \}$ are vertex-disjoint red copies of $K_4$.
    
    Suppose $|f(V)| =2$. Without loss of generality, support that $f(V) = \{n_1,n_2\}$ and $|f^{-1}(n_1)|\ge 3$. Because there are no blue copies of $K_4$ in our graph, $f^{-1}(n_1)$ contains a red edge $\{v_i,v_j\}$, and the subsets $\{v_i,v_j,n_2,n_3\}$ and $\{v_k,n_1,n_4,n_5\}$ are vertex-disjoint red copies of $K_4$, where $v_k \in f^{-1}(n_2)$.
    
    Suppose $|f(V)|=1$. Without loss of generality, suppose $f(V) = \{n_1\}$. Then $V$ does not contain a blue copy of $K_3$, otherwise our graph would contain a blue copy of $K_4$. If $V$ contains a red copy $\{v_1,v_2,v_3\}$ of $K_3$, then $\{v_1,v_2,v_3,n_2\}$ and $\{v_4,n_3,n_4,n_5\}$ are two red copies of $K_4$, and we are done. If $V$ does not contain a red or blue copy of $K_3$, then the red edges in $V$ form a cycle of length five, and there are two vertex-disjoint red edges in $V$, denoted $\{v_i,v_j\}$ and $\{v_k,v_l\}$. In this case, the subsets $\{v_i,v_j,n_2,n_3\}$ and $\{v_k,v_l,n_4,n_5\}$ are both red copies of $K_4$.

    \item[Case II:] Suppose that there exist at most four vertices in $K_{20}\backslash N$ with at least three red edges adjacent to $N$. Then there are at least eleven vertices in $K_{20}\backslash N$ with at least three blue edges adjacent to $N$. Because no vertex has two red and two blue edges adjacent to a red copy of $K_4$, these vertices have at least four blue edges adjacent to $N$, and so there exists a vertex $n_i \in N$ with at least nine blue edges adjacent to $K_{20}\backslash N$. Therefore, $K_{20} \backslash N$ must contain a red copy of $K_4$.
\end{itemize}
This completes the proof.
\end{proof}

\end{document}